\documentclass[10pt]{amsart}
\usepackage{amssymb,amsbsy,amsmath,amsfonts,times, amscd}
\usepackage{latexsym,euscript,exscale}
\usepackage{graphicx,color} \usepackage{amsthm}
\usepackage{fancyhdr} \usepackage{url}
\usepackage{epigraph}\usepackage{lmodern}
\usepackage{mathrsfs}\usepackage{cancel}
\usepackage[toc,page]{appendix}\usepackage[T1]{fontenc}
\usepackage{mathtools}\usepackage[all,cmtip]{xy}
\usepackage{enumerate}\usepackage{setspace}
\usepackage{helvet}
\usepackage{float}
\usepackage{hyperref}

\numberwithin{equation}{section}
\newtheorem{theorem}{Theorem}[section]
\newtheorem*{theorem*}{Theorem}
\newtheorem{proposition}[theorem]{Proposition}
\newtheorem*{proposition*}{Proposition}
\newtheorem{lemma}[theorem]{Lemma}
\newtheorem*{lemma*}{Lemma}
\newtheorem{corollary}[theorem]{Corollary}
\newtheorem*{corollary*}{Corollar}

\newtheorem*{fact*}{Fact}
\theoremstyle{definition}
\newtheorem{definition}[theorem]{Definition}
\newtheorem*{definition*}{Definition}

\newtheorem*{claim*}{Claim}

\newtheorem*{conjecture*}{Conjecture}

\newtheorem{problem}[theorem]{Problem}

\newtheorem{example}[theorem]{Example}
\newtheorem*{example*}{Example}
\newtheorem{remark}[theorem]{Remark}

\newtheorem*{remark*}{Remark}

\newtheorem*{note*}{Note}
\newtheorem*{question*}{Question}

\newcommand {\D}{\mathbb D}

\newcommand {\M}{\mathbb M}
\newcommand {\C}{\mathbb C}
\newcommand {\Z}{\mathbb Z}

\newcommand{\eps}{\varepsilon}

\newcommand{\Q}{\mathbb{Q}}
\newcommand{\N}{\mathbb{N}}
\newcommand{\R}{\mathbb{R}}

\newcommand{\Lip}{\mathrm{Lip}}
\newcommand{\diam}{\mathrm{diam}}

\begin{document}

 \title[On the expansiveness of coarse maps between Banach spaces]{On the expansiveness of coarse maps between Banach spaces and geometry preservation}
 \author[B. M. Braga]{Bruno M. Braga}
\address[B. M. Braga]{IMPA, Estrada Dona Castorina 110, 22460-320, Rio de Janeiro, Brazil.}
\email{demendoncabraga@gmail.com}
\urladdr{https://sites.google.com/site/demendoncabraga/}
\thanks{B. M. Braga  was partially supported by FAPERJ (Proc. E-26/200.167/2023) and by CNPq (Proc. 303571/2022-5).}
 
 \author[G. Lancien]{Gilles Lancien}
\address[G. Lancien]{Laboratoire de Math\'ematiques de Besan\c con, Universit\'e de Franche-Comt\'e, CNRS UMR-6623, 16 route de Gray, 25030 Besan\c con C\'edex, Besan\c con, France}
\email{gilles.lancien@univ-fcomte.fr}
\urladdr{https://lmb.univ-fcomte.fr/Lancien-Gilles}

\maketitle

\begin{abstract}
We introduce a new notion of embeddability between Banach spaces. By studying the classical Mazur map, we show that it is strictly weaker than the notion of coarse embeddability. We use the techniques from metric cotype introduced by M. Mendel and A. Naor to prove results about cotype preservation and complete our study of embeddability between $\ell_p$ spaces. We confront our notion with nonlinear invariants introduced by N. Kalton, which are defined in terms of concentration properties for Lipschitz maps defined on countably branching Hamming or interlaced graphs. Finally, we address the problem of the embeddability into $\ell_\infty$.
\end{abstract}

\section{Introduction}

This article deals with a new notion of nonlinear embeddability between Banach spaces and how their geometries are preserved under this new notion. More precisely, the notion considered herein will be large scale in nature and even weaker than the usual  coarse embeddability. Before presenting it, we start by recalling the basics of coarse geometry. Given metric spaces $(X,d)$ and $(Y,\partial)$, and a map $f\colon X\to Y$, one defines a modulus
\[\omega_f(t)=\sup\{\partial(f(x),f(z))\mid d(x,z)\leq t\},\ \text{ for }\ t\geq 0,\]
and call $f$ \emph{coarse} if $\omega_f(t)<\infty$ for all $t\geq 0$. In  words,  $f$ is coarse if it sends bounded sets to bounded sets in a uniform manner. Coarse maps are the usual morphisms considered in the study of the large scale geometry of metric spaces and, in particular, of Banach spaces. In order to deal with embeddings, one defines a modulus 
\[\rho_f(t)=\inf\{\partial (f(x),f(z))\mid d(x,z)\geq t\},\ \text{ for } \ t\geq 0,\]
and call $f$ \emph{expanding} if $\lim_{t\to \infty}\rho_f(t)=\infty$. In words, $f$ is expanding if it sends elements far apart to elements likewise uniformly. The map $f$ is then called a \emph{coarse embedding} if it is both coarse and expanding. Despite its seemingly weak definition, coarse embeddability is known to capture the geometry of Banach spaces in several remarkable ways; to cite a few, we mention the cotype preservation under coarse embeddability into Banach spaces with nontrivial type proved in the seminal paper of M. Mendel and A. Naor (\cite[Theorem 1.11]{MendelNaor2008}) and the preservation of asymptotic-$c_0$-ness$+$reflexivity proved by the second named author together with F. Baudier, P. Motakis, and Th. Schlumprecht (\cite[Theorem A]{BaudierLancienSchlumprecht2018}). 

Functional analysts working in the nonlinear geometry of Banach spaces are interested in knowing the minimal requirements needed for maps between Banach spaces to still generate an interesting notion of embeddability; here the word ``interesting'' should be broadly interpreted as ``it is strictly weaker than a previously studied notion of embeddability but still strong enough to impose geometric restrictions''. For instance, C. Rosendal has started in \cite{Rosendal2017Sigma} with the program of weakening the notion of expansiveness of a coarse map $f$ by properties such as $f$ being \emph{uncollapsed} in the sense that there are $\Delta,\delta>0$ such that 
\[\|x-z\|\geq \Delta\ \text{ implies }\ \|f(x)-f(z)\|>\delta,\]
or $f$ being \emph{solvent}, meaning that  there is an increasing sequence $(R_n)_n$ in $\N$ such that 
\[\|x-z\|\in [R_n,R_n+n]\ \text{ implies }\ \|f(x)-f(z)\|>n.\]
Even maps $f$ satisfying only that 
\[\|x-z\|=\Delta\ \text{ implies }\ \|f(x)-f(z)\|>\delta\]
have already been studied; those are called \emph{almost uncollapsed} (see \cite{Braga2017JFA}). Inspired by a recent work by the two authors (see \cite{BragaLancien2023Equiv}), this paper initiates a yet new approach of weakening the expansiveness condition. Indeed, all the weakenings mentioned above are \emph{not} localized: the positions of $x$ and $z$ in $X$ do not matter, but only the distance $\|x-z\|$. However, in \cite{BragaLancien2023Equiv}, the authors started the study of an equivalence between metric spaces called \emph{asymptotic coarse equivalence} and this takes into account the asymptotic behavior of elements $x$ in $X$ as they approach infinity. In particular, those maps are not necessarily expanding anymore, but only satisfy expansiveness as $x,z\to \infty$. This motivates the main definition of these notes:

\begin{definition}
Let $X$ and $Y$ be Banach spaces and $\alpha\in [0,1]$. A map $f\colon X\to Y$ is called \emph{expanding at rate $\alpha$} if for all $L>0$ there is a map $\rho\colon [0,\infty)\to [0,\infty)$ with $\lim_{t\to \infty}\rho(t)=\infty$ such that 
\[\|x-z\|\geq L\max\{\|x\|^\alpha,\|z\|^\alpha\}+L\ \text{ implies }\ \|f(x)-f(z)\|\geq  \rho(\|x-z\|).\] 
In case $\rho$ can always be chosen to be of the form $\frac{t}{C}-C$ for some $C>0$, we say that $f$ is \emph{linearly expanding at rate $\alpha$}.
\end{definition}

A few comments are in place here. Firstly, notice that a coarse map $X\to Y$ is expanding if and only if it is expanding at rate $0$. Also, we restrict ourselves to $\alpha\leq 1$ since the condition of $\|x-z\|$ being at least of the order of $\max\{\|x\|^\alpha,\|z\|^\alpha\}$ will not happen (up to a bounded subset) if $\alpha>1$.  We   say that a coarse map $f$ has \emph{nontrivial coarse expansion} if it is expanding at rate $\alpha$ for some $\alpha\in [0,1]$. Finally, we recall that if a coarse map $f$ between Banach spaces satisfies $\rho_f(t)\geq \frac{t}{C}-C$ for some $C>0$ and  all $t\geq 0$, then $f$ is called a \emph{coarse Lispchitz embedding}; a stronger notion than coarse embedding. Hence, the notion of a coarse map $f$ being linearly expanding at some rate should be seen as a weakening of $f$ being a coarse Lipschitz embedding.  We now describe the main findings of this paper.

\subsection{Mazur maps and metric cotype}
We first show that the existence of coarse maps which are expanding at rate $\alpha$, for $\alpha\in (0,1]$, is strictly weaker than coarse embeddability. In fact, as it turned out, there are several  well-studied maps which witness that: the Mazur maps. Recall, given $p,q\in [1,\infty)$, the \emph{Mazur map} $M_{p,q}\colon \ell_p\to \ell_q$ is the homogeneous extension of the   canonical map which adjusts elements in the unit sphere of $\ell_p$ so that they fall into the unit sphere of $\ell_q$; for brevity, we postpone to Section \ref{SectionMazur}  its formal definition. While $\ell_p$ coarsely embeds into $\ell_q$ if and only if either $p\in [1,2]$ or $p\leq q$ (\cite[Corollary 7.3]{MendelNaor2008}), we show that $M_{p,q}$ is coarse and has nontrivial expansion as long as $p>q$. Precisely:

\begin{theorem}\label{ThmMazurMapACE}
Let $1\le q<p$ and $\alpha\in ( \frac{p-q}{p},1]$. Then $M_{p,q}\colon \ell_p\to \ell_q$ is a coarse map, which is expanding at rate $\alpha$. Moreover, $M_{p,q}$ is linearly expanding at rate $1$. 
\end{theorem}

If $p>2$ and $p>q$, we know from the aforementioned result of M. Mendel and A. Naor that $\ell_p$ does not coarsely embeds into $\ell_q$; in our terminology just introduced, this means that there is no coarse map $\ell_p\to \ell_q$ which is expanding at rate $0$. By Theorem \ref{ThmMazurMapACE}, we are then left to understand what happens for $\alpha$'s in the interval $(0,\frac{p-q}{p}]$. We show, using techniques from metric cotype of \cite{MendelNaor2008}, that there is no such map for all $\alpha$'s in $(0,\frac{p-q}{p}]$, when $2\le q< p<\infty$,  and that there is no such map for all $\alpha$'s in $(0,\frac{p-2}{p}]$, when $1\le q<2<p$ (see Corollary \ref{CorEmblpIntolq}). We refer to Section \ref{SectionCotype} for all relevant definitions. Let us just say for this introduction that, for a Banach space $X$, we denote \[q_X=\inf\{q\in [2,\infty]\mid X \text{ has cotype q}\}.\] We also prove the following general result about cotype preservation, which generalizes \cite[Theorem 1.11]{MendelNaor2008}:

  \begin{theorem}\label{ThmCotypeACE}
 Let $X$ and $Y$ be Banach spaces and suppose $Y$ has nontrivial type. Let $\alpha\in [0,1]$ and suppose there is a coarse map $X\to Y$ which is expanding at rate $\alpha$. Then $q_X\leq \frac{q_Y}{1-\alpha}$.
 \end{theorem}

We stress here that our results do more than generalizing \cite[Theorem 1.11]{MendelNaor2008}. Indeed, Theorem \ref{ThmCotypeACE} together with Theorem \ref{ThmMazurMapACE} give us optimal results on cotype preservation  (see Corollary \ref{CorEmblpIntolq}).

\subsection{Embeddings of metric graphs into Banach spaces}
After our study of cotype preservation and the embeddability of the $\ell_p$'s, we turn our attention to the embeddability of certain metric graphs into Banach spaces. Recall, given $k\in\N$, we let $[\N]^k$ denote the set of all subsets of $\N$ with $k$ elements and,  given $\bar n\in[\N]^k$, we   write $\bar n=(n_1,\ldots, n_k)$ where  $n_1<\ldots< n_k$. As initiated by  N. Kalton (see \cite{Kalton2007,Kalton2013AsymptoticStructure}), the study of the embeddability of the sequence $([\N]^k)_k$ (endowed with appropriate metrics) is extremely useful when looking for coarse (or coarse Lipschitz) invariants of Banach spaces. For instance,  given $k\in\N$, let $d_{\mathbb H}=d_{\mathbb H,k}$ denote the \emph{Hamming metric}  on $[\N]^k$, i.e., 
\[d_{\mathbb H}(\bar n, \bar m)=|\{i\in \{1,\ldots, k\}\mid n_i\neq m_i\}|, \ \text{ for }\ \bar n,\bar m\in [\N]^k. \]

The following important concentration property was introduced in \cite{KaltonRandrianarivony2008} and later formalized by A. Fovelle in \cite{Fovelle} in the format presented below:
\begin{definition}\label{DefHCF}
Let $p\in (1,\infty]$.    A Banach space $X$ is said to have \emph{Hamming full concentration property $p$}, abbreviated \emph{HFC}$_p$, if there is $C\geq 1$ such that for all $k\in\N$ and all  $1$-Lipschitz maps $\phi\colon ([\N]^k,d_{\mathbb H})\to X$, there is an infinite $\M\subseteq \N$ such that 
\[\mathrm{diam}(\phi([\M]^k))\leq Ck^{1/p}\]
(here we use the convention $1/\infty=0$ if $p=\infty$). \end{definition}

As shown in \cite[Theorem 4.2]{KaltonRandrianarivony2008}, reflexive spaces with a $p$-asymptotically uniformly smooth renorming have the HFC$_p$; see Example \ref{example HFC} for definitions. Moreover, as shown in \cite[Theorems A and B]{BaudierLancienMotakisSchlumprecht2018}, having HFC$_\infty$ is equivalent to $X$ being asymptotic-$c_0$ and reflexive. We postpone to Example \ref{ExaAsympc_0} the formal definition of asymptotic-$c_0$-ness, for now, we simply say that $X$ has such property if copies of the finite dimensional subspaces of $c_0$ can be found in the finite codimensional subspaces of $X$ in a uniform manner.

It was known that for $p\in (1,\infty)$, HFC$_p$ is stable under coarse-Lipschitz embeddings and that HFC$_\infty$ is even stable under coarse embeddings.  We show that our weaker notions of embeddability are already enough for the HFC$_p$ properties to be preserved in the following sense:

\begin{theorem}\label{ThmPreservationHFCpByLinearExpansion} Let $X$ and $Y$ be Banach spaces and suppose $Y$ has HFC$_p$ for some $p\in (1,\infty]$. 
\begin{enumerate}
    \item\label{ThmPreservationHFCpByLinearExpansionItem1} Suppose $p\in (1,\infty)$. If there is a coarse map $f\colon X\to Y$ which is linearly expanding at rate $1$, then  $X$ must have HFC$_p$.

    \item\label{ThmPreservationHFCpByLinearExpansionItem2} Suppose $p=\infty$. If there is a coarse map $f\colon X\to Y$ which is  expanding at rate $1$, then  $X$ must have HFC$_\infty$.

\end{enumerate}
\end{theorem}

In particular, applying Theorem \ref{ThmPreservationHFCpByLinearExpansion}\eqref{ThmPreservationHFCpByLinearExpansionItem2} together with the characterization of HFC$_\infty$ mentioned above (\cite[Theorems  B]{BaudierLancienMotakisSchlumprecht2018}), yields immediately the following corollary.

\begin{corollary}\label{CorollaryHFCInfty}
 If a Banach space $X$ can be mapped by a coarse map which is also expanding at rate $1$ into  a reflexive Banach space which is   asymptotic-$c_0$, then $X$ must be also  reflexive and  asymptotic-$c_0$.
\end{corollary}

We also study interlacing pairs in Hamming graphs in Subsection \ref{SubsectionInterlacingHamming} and use this to obtain results about the embeddability of the James spaces $J_p$ (Example \ref{ExampleJamesSpace}); see Theorem \ref{ThmInterlacingHamming} and Corollary \ref{CorInterlacingHammingJames} for details. 

\medskip
Another important metric we can endow each $[\N]^k$ with is the \emph{interlaced metric}: we set distinct elements $\bar n,\bar m\in [\N]^k$ to be adjacent if either 
\[n_1\leq m_2\leq \ldots\leq n_k\leq m_k\ \text{ or }\ m_1\leq n_1\leq \ldots\leq m_k\leq n_k\]
and then we let $d_{\mathbb I}=d_{\mathbb I,k}$ be the shortest path metric on $[\N]^k$ given by this graph structure. The study of those metric spaces was fundamental for N. Kalton to rule out the coarse embeddability of $c_0$ into reflexive spaces and gave rise to the so-called \emph{property $\mathcal Q$'s}:

\begin{definition}\label{DefiPropQ}
Let $p\in (1,\infty]$.    A Banach space $X$ is said to have \emph{property $\mathcal Q_p$}, if there is $C\geq 1$ such that for all $k\in\N$ and all  $1$-Lipschitz maps $\phi\colon ([\N]^k,d_{\mathbb I})\to X$, there is an infinite $\M\subseteq \N$ such that 
\[\mathrm{diam}(\phi([\M]^k)\leq Ck^{1/p}\]
(here we use the convention $1/\infty=0$ if $p=\infty$). If $p=\infty$, we simply say $X$ has \emph{property $\mathcal Q$}. \end{definition}

We prove that our weakenings of coarse embeddability are also strong enough to ensure the preservation of property $\mathcal Q_p$. Precisely, we prove the following:

\begin{theorem}\label{ThmPreservationPropertyQpByLinearExpansion} Let $X$ and $Y$ be Banach spaces and suppose  $Y$ has property $\mathcal Q_p$ for some $p\in (1,\infty]$.  

\begin{enumerate}
    \item\label{ThmPreservationPropertyQpByLinearExpansionItem1} Suppose $p\in (1,\infty)$. If there is a coarse map $f\colon X\to Y$ which is linearly expanding at rate $1$, then  $X$ must have property $\mathcal Q_p$.

    \item\label{ThmPreservationPropertyQpByLinearExpansionItem2} Suppose $p=\infty$. If there is a coarse map $f\colon X\to Y$ which is  expanding at rate $1$, then  $X$ must have property $\mathcal Q_\infty$, i.e., property $\mathcal Q$.
\end{enumerate}
\end{theorem}

In Section \ref{SectionlInfty}, we characterize Lipschitz embeddability into $\ell_\infty$ in terms of our new notion of embeddability;  this extends a result of N. Kalton (see \cite[Theorem 5.3]{Kalton2011FundMath}). Precisely, we prove the following.
 
\begin{theorem}\label{ThmlInfty}
The following are equivalent for a Banach space $X$:
\begin{enumerate}
\item\label{ThmlInftyItem1} $X$ Lipschitzly embeds into $\ell_\infty$.
\item\label{ThmlInftyItem2} There is a  coarse map $f\colon X\to \ell_\infty$ which is linearly expanding at rate $\alpha$, for some $\alpha\in (0,1)$.
\end{enumerate}
\end{theorem}

 
\section{Revisiting the Mazur map}\label{SectionMazur}

The purpose of this section is to prove  some estimates for the classic Mazur map (see Lemma \ref{LemmaMazupMapEstimate}) and to deduce the proof of Theorem \ref{ThmMazurMapACE}. Recall,  the Mazur map between the unit spheres of two Lebesgue sequence spaces, say $\ell_p$ and $\ell_q$, is the canonical map which adjusts  $p$-summable sequences so that they are $q$-summable. Precisely, given a Banach space $X$, $B_X$ denotes its closed unit ball and $\partial B_X$ its unit sphere. Then, given $p,q\in [1,\infty)$, the \emph{Mazur map} is the map $M_{p,q}\colon \partial B_{\ell_p}\to\partial  B_{\ell_q}$ given by 
\[M_{p,q}((x_n)_n)=(\mathrm{sign}(x_n)|x_n|^{p/q})_n\]
for all $(x_n)_n\in \partial B_{\ell_p}$. It is evident that $M_{p,q}$ is a bijection with inverse $M_{q,p}$. Moreover,  it is well-known that this map is a uniform equivalence. Precisely, the following estimates hold (see \cite[Theorem 9.1]{BenyaminiLindenstraussBook} for a proof and \cite{MazurStudia1929} for its first appearance).

\begin{theorem}
[Mazur Map]
Let $p,q\in [1,\infty)$ with $q<p$ and let $M_{p,q}\colon \ell_p\to \ell_q$ denote the Mazur map. There is a constant $C=C(p,q)>0$ such that 
\[C\|x-y\|^{p/q}\leq \|M_{p,q}(x)-M_{p,q}(y)\|\leq \frac{p}{q}\|x-y\|\]
for all $x,y\in \partial B_{\ell_p}$.\label{ThmMazurMap}
\end{theorem}

As it is  usually done, we extend the map $M_{p,q}$ to the whole $\ell_p$ by homogeneity; by a abuse of notation, we still denote this extension by $M_{p,q}$. Precisely, for each $x\in \ell_p$, we let 
\[M_{p,q}(x)=\left\{\begin{array}{ll}
\|x\|M_{p,q}\Big(\frac{ x}{\| x\|}\Big),& \text{ if }\ x\neq 0,\\
0,& \text{ if }\ x=0.
\end{array}   \right.\]
We emphasize some important properties of the Mazur map below:
\begin{itemize}
\item    $M_{p,q}^{-1}=M_{q,p}$ for all $p,q\in [1,\infty)$, 
\item $\|M_{p,q}(x)\|=\|x\|$ for all $x\in \ell_p$, and
\item $M_{p,q}\Big(\frac{x}{\|x\|}\Big)=\frac{M_{p,q}(x)}{\|M_{p,q}(x)\|}$ for all $x\in \ell_p\setminus \{0\}$. 
\end{itemize}  
 
The following is the main technical result of this section.
 
\begin{lemma}\label{LemmaMazupMapEstimate}
Let $p,q\in [1,\infty)$ with $q<p$. 
\begin{enumerate}
\item\label{Item1PropMazupMapEstimate} For all $x,y\in \ell_p$, we have   
\[\|M_{p,q}(x)-M_{p,q}(y)\|\leq \Big(\frac{2p}{q}+1\Big)\|x-y\|.\]


\item\label{Item3PropMazupMapEstimate} For all $\eps>0$ and all $\alpha\in (0,1]$, there is $L=L(p,q,\eps,\alpha)>0$ so that, for all  $x,y\in \ell_p$ with   $\|x-y\|\geq \eps \max\{\|x\|^\alpha,\|y\|^\alpha\}$, we have  
\[\frac{1}{L}\|x-y\|^{\frac{p}{q}-\frac{1}{\alpha}\big(\frac{p}{q}-1\big)}\leq \|M_{p,q}(x)-M_{p,q}(y)\|.\]
\end{enumerate}
\end{lemma}

Before proving Lemma \ref{LemmaMazupMapEstimate}, we  isolate some simple estimates  for future use  (cf. \cite[Lemma 3.1]{Kalton2013Examples}).

\begin{lemma}\label{LemmaEstimates}
Let $X$ be a Banach space and $x,y\in X$ with $\|x\|\geq \|y\|>0$. The following holds
\begin{enumerate}
\item $\big\|\frac{x}{\|x\|}-\frac{y}{\|y\|}\big\|\leq 2\frac{\|x-y\|}{\|x\|}$ and
\item $\|x-y\|\leq \|x\|-\|y\|+\|y\| \big\|\frac{x}{\|x\|}-\frac{y}{\|y\|}\big\|$.\qed
\end{enumerate}
\end{lemma}

\begin{proof}[Proof of Lemma \ref{LemmaMazupMapEstimate}]
\eqref{Item1PropMazupMapEstimate} Fix $x,y\in \ell_p$. Without loss of generality,   suppose $\|x\|\geq \|y\|>0$. Then   Theorem \ref{ThmMazurMap} and Lemma \ref{LemmaEstimates} imply
\begin{align*}
\|M_{p,q}(x)-M_{p,q}(y)\|= & \Big\| \|x\|M_{p,q}\Big(\frac{ x}{\| x\|}\Big)- \|y\|M_{p,q}\Big(\frac{ y}{\| y\|}\Big)\Big\|
\\
\leq &\|x\|\Big\| M_{p,q}\Big(\frac{ x}{\| x\|}\Big)- M_{p,q}\Big(\frac{ y}{\| y\|}\Big)\Big\|\\ 
+ & \Big(\|x\|-\|y\|\Big)\Big\| M_{p,q}\Big(\frac{ y}{\| y\|}\Big)\Big\|\\
\leq & \Big(\frac{2p}{q}+1\Big)\|x-y\|.
\end{align*}


\eqref{Item3PropMazupMapEstimate}
 Fix $\eps>0$, $\alpha\in (0,1]$, and  $x,y\in \ell_p$ with $\|x-y\|\geq \eps\max\{\|x\|^\alpha,\|y\|^\alpha\}$. It is easily checked that we may  assume that $\|x\|\geq \|y\|>0$.
 
Suppose first that  $\|x\|-\|y\|>\frac{\|x-y\|}{2}$, then
\begin{align*}
\|M_{p,q}(x)-M_{p,q}(y)\|&\geq \|M_{p,q}(x)\|-\|M_{p,q}(y)\|\\
&= \|x\|-\|y\|\\
&\geq \frac{\|x-y\|}{2}
\end{align*}
(here we use that $\|M_{p,q}(z)\|=\|z\|$ for all $z\in \ell_p$). 

Suppose now that $0\le \|x\|-\|y\|\le \frac{\|x-y\|}{2}$. Then, by Lemma \ref{LemmaEstimates}, we have 
 \begin{equation}\label{Eq.Lemma.Est1} \|x-y\|\leq  2\|y\|\Big\|\frac{x}{\|x\|}-\frac{y}{\|y\|}\Big\|\leq 2\|x\|\Big\|\frac{x}{\|x\|}-\frac{y}{\|y\|}\Big\|.
 \end{equation}
Using Theorem \ref{ThmMazurMap} and the fact that    $\|M_{p,q}(z)\|=\|z\|$ for all $z\in \ell_p$ again, we have
  \begin{align*}
\|x\|\Big\|\frac{x}{\|x\|}-\frac{y}{\|y\|}\Big\|&\leq \frac{\|x\|}{C^{q/p}}\Big\|M_{p,q}\Big(\frac{x}{\|x\|}\Big)-M_{p,q}\Big(\frac{y}{\|y\|}\Big)\Big\|^{q/p}\\
&= \frac{\|x\|}{C^{q/p}}\Big\|\frac{M_{p,q}(x)}{\|M_{p,q}(x)\|}-\frac{M_{p,q}(y)}{\|M_{p,q}(y)\|}\Big\|^{q/p} \notag\\
&\leq \Big(\frac{2}{C}\Big)^{q/p}\|x\|^{1-q/p}\|M_{p,q}(x)-M_{p,q}(y)\|^{q/p}. \notag
 \end{align*}
As $q<p$ and   $\|x-y\|\geq\eps \max\{\|x\|^\alpha,\|y\|^\alpha\}=\eps \|x\|^\alpha$, \eqref{Eq.Lemma.Est1} and the inequality above imply that
  \begin{align}\label{Eq.1}
 \|x-y\|\leq 2\Big(\frac{2}{C}\Big)^{q/p}\eps^{\frac{q/p-1}{\alpha}}\|x-y\|^{\frac{1-q/p}{\alpha}}\|M_{p,q}(x)-M_{p,q}(y)\|^{q/p}.\notag
 \end{align}
Simplifying the above, we conclude that 
 \[\|x-y\|^{\frac{p}{q}-\frac{1}{\alpha}\big(\frac{p}{q}-1\big)}\leq  \frac{2^{1+p/q}}{C}  \eps^{\frac{1-p/q}{\alpha}}\|M_{p,q}(x)-M_{p,q}(y)\|.\]
 The result then follows taking  $L$ to be the maximum of $2$ and  $C^{-1}2^{1+p/q} \eps^{(1-p/q)/\alpha}$.
 \end{proof}

\begin{proof}[Proof of Theorem \ref{ThmMazurMapACE}]
   The first claim  is immediate from Lemma \ref{LemmaMazupMapEstimate} since $\alpha$ being larger than $\frac{p-q}{p}$ implies that $\frac{p}{q}-\frac{1}{\alpha}\big(\frac{p}{q}-1\big)$  is positive. The second claim follows  immediately from Lemma \ref{LemmaMazupMapEstimate} also since   $\frac{p}{q}-\frac{1}{\alpha}\big(\frac{p}{q}-1\big)=1$ if $\alpha=1$.  
\end{proof}

As a consequence of Theorem  \ref{ThmMazurMapACE}, we obtain that the existence of a coarse map between Banach spaces which is expanding at rate $\alpha$, for $\alpha\in (0,1]$, is strictly weaker than coarse embeddability. Indeed, it  contrasts with the well-known fact that   $\ell_p$ coarsely embeds into $\ell_q$ if and only if either $ p\leq q$ or $p,q\in [1,2]$ (see \cite[Corollary 7.3]{MendelNaor2008}).


\section{Expansion  and cotype preservation}\label{SectionCotype}

In Section \ref{SectionMazur}, we showed that, for  $p,q\in [1,\infty)$ with $q<p$,  there are coarse maps with nontrivial expanding properties from $\ell_p$ to $\ell_q$; which contrasts with the known results about the coarse embeddability of  $\ell_p$ into $\ell_q$. Precisely, Theorem \ref{ThmMazurMapACE} shows there are coarse maps $\ell_p \to \ell_q$ which are expanding at rate $\alpha$ as long as $\alpha> \frac{p-q}{p}$. Since we know that $\ell_p$ coarsely embeds into $\ell_q$ if and only if either $p\leq q$ or $p,q\in [1,2]$ (see \cite[Corollary 7.3]{MendelNaor2008}), this result is not true for $\alpha=0$. The main goal of the current    section is to  understand what happens for $\alpha$'s in the interval $(0,\frac{p-q}{p}]$. Using techniques developed by M. Mendel and A. Naor in their seminal paper about metric cotype (\cite{MendelNaor2008}), we show that the results of Section \ref{SectionMazur} do not hold for  $2\le q<p<\infty$ and $\alpha\le \frac{p-q}{p}$, nor for $1\le q<2<p$ and $\alpha \le \frac{p-2}{p}$. 

\medskip
We start this section recalling the necessary background on metric cotype. Given $m\in\N$, $\Z_m$ denotes the set of integers modulo $m$. Given $n,m\in\N$, $\mu=\mu_{m,n}$ denotes the normalized counting measure on $\Z^n_m$ and $\sigma=\sigma_n$ denotes the normalized counting measure on $\{-1,0,1\}^n$. For each $j\in \{1,\ldots, n\}$, $e_j$ denotes the vector in $\Z^n_m$ whose $j$-th coordinate is $1$ and all others are $0$. 

\begin{definition}[Metric cotype] Let $(X,d)$ be a metric space and $q,\Gamma>0$. We say that $X$ has \emph{metric cotype $q$ with constant $\Gamma$} if for all $n\in\N$ there is an even $m\in\N$ such that for all $f\colon \Z^n_m\to X$ we have 
\begin{align}
\label{Eq.MetricCotype}
\sum_{j=1}^n\int_{\Z^n_m}d\Big(f\big(x &+\frac{m}{2}e_j\big),f(x)\Big)^qd\mu(x)\\
&\leq \Gamma^qm^q \int_{\{-1,0,1\}^n}\int_{\Z^n_m}d\big(f(x+\eps),f(x)\big)^qd\mu(x)d\sigma(\eps).\notag
\end{align}
 Given $n\in\N$ and $\Gamma>0$, we let  $m_q(X, n, \Gamma)$ be the 
smallest even integer $m$ such that \eqref{Eq.MetricCotype} holds for all $f \colon  \Z^n_m\to X$. If no such $m$ exists, we set $m_q(X, n,\Gamma) = \infty$.
\end{definition}

The next lemma is the main technical result of this section and is a refinement of \cite[Lemma 7.1]{MendelNaor2008}. 

\begin{lemma}\label{LemmaCotypeInequality}
Let $(X,d)$ be a metric space, $n\in\N$, $q,s,\Gamma>0$, and $r\in (0,\infty]$. Let $\alpha\in [0,\frac{rs}{rs+1}]$,   $\rho\colon [0,\infty)\to [0,\infty)$, and  maps $f_n\colon \ell_r^n(\C)\to X$, for $n\in \N$,  be such that 
\[\|x-z\|\geq \max\{\|x\|^\alpha,\|z\|^\alpha\}+1\ \text{ implies }\ \|f_n(x)-f_n(z)\|\geq \rho(\|x-z\|),\] 
for all $n\in \N$ and $x,z \in \ell_r^n(\C)$. Then, we have that for all $n\in \N$:
\[n^{1/q}\rho(2n^s)\leq \Gamma\cdot m_q(X,n,\Gamma)\cdot \omega_{f_n}\Big(\frac{2\pi n^{s+1/r}}{m_q(X,n,\Gamma)}\Big)\]
(if $r=\infty$, we use the conventions $1/\infty=0$ and $\infty/\infty=1$).
\end{lemma}
 
 \begin{proof}
To simplify notation, let $m=m_q(X,n,\Gamma)$. By a slight abuse of notation, we let $e_1,\ldots, e_n$ denote the standard basis of both $\Z^n_m$ and  $\ell_r^n(\C)$. Define a map $h\colon \Z_m^n\to \ell_r^n(\C)$ by letting 
\[h(x)=n^s\cdot\sum_{j=1}^n e^{\frac{2\pi i x_j}{m}}e_j\]
for all $x=(x_j)_{j=1}^n\in \Z_m^n$. Note first that, since $t\mapsto e^{it}$ is $1$-Lipschitz on $\R$,  
\[\|h(x+\eps)-h(x)\|_r \le n^s\Big\|\Big(\frac{2\pi |\eps_j|}{m}\Big)_{j=1}^n\Big\|_r \le \frac{2\pi n^{s+1/r}}{m},\]
for all $\eps=(\eps_j)_{j=1}^n\in \{-1,0,1\}^n$ and all $x=(x_j)_{j=1}^n\in Z_m^n$. Let $g_n=f_n\circ h$, then 
\[\|g_n(x+\eps)-g_n(x)\|\leq \omega_{f_n}\Big(\frac{2\pi n^{s+1/r}}{m}\Big)\]
for all $\eps=(\eps_j)_{j=1}^n\in \{-1,0,1\}^n$ and all $x=(x_j)_{j=1}^n\in Z_m^n$. Therefore, integrating with respect to $x$ and $\eps$, we have 
\begin{equation}\label{Eq.cotype1}
\int_{\{-1,0,1\}^n}\int_{\Z_m^n}\|g_n(x+\eps)-g_n(x)\|^qd\mu(x)d\sigma(\eps)\leq \omega_{f_n}\Big(\frac{2\pi n^{s+1/r}}{m}\Big)^q.
\end{equation}
We now use the hypothesis on $f_n$. As $\alpha\le \frac{rs}{rs+1}$, we have that    
\[2n^s=2(n^{s+1/r})^{\frac{rs}{rs+1}}\geq  (n^{s+1/r})^\alpha+1,\] 
for all $n\in \N$. Hence,  
\begin{align*}
\Big\|h\Big( x+\frac{m}{2}e_j\Big)-h(x)\Big\|_r=2n^{s} \ge (n^{s+1/r})^\alpha+1, 
\end{align*}	
for all $x\in \Z_m^n$ and all $j\in \{1,\ldots, n\}$. Therefore, as $\|h(y)\|_r=n^{s+1/r}$ for all $y\in \Z_m^n$, it follows that
\[ \Big\|g_n\Big(x+\frac{m}{2}e_j\Big)-g_n(x)\Big\|\geq \rho(2n^{s})\]
for all $x\in \Z_m^n$ and all $j\in \{1,\ldots, n\}$. Hence, 
\begin{equation}\label{Eq.cotype2}\sum_{j=1}^n\int_{\Z_m^n}\Big\|g_n\Big(x+\frac{m}{2}e_j\Big)-g_n(x)\Big\|^qd\mu(x)\geq n\rho(2n^{s})^q.
\end{equation}
By the definition of $m=m_q(M,n,\Gamma)$, \eqref{Eq.cotype1} and \eqref{Eq.cotype2} show that  
\[n\rho(2n^{s})^q\leq \Gamma^qm^q \omega_{f_n}\Big(\frac{2\pi n^{s+1/r}}{m}\Big)^q.\]
Taking the $q$-th root at both sides above  finishes the proof.
\end{proof}
 
 We now turn to our result about preservation of cotype by coarse maps satisfying some weak expanding conditions. For completeness, we quickly recall the notions of type and cotype. Let $X$ be a Banach space and $p\in (1,2]$. We say that $X$ has \emph{type $p$}  if there is $C>0$ such that 
 \[ \frac{1}{2^n} \sum_{(\eps_i)_{i=1}^n\in \{-1,1\}^n}\Big\|\sum_{i=1}^n\eps_ix_i\Big\|^p\leq C\sum_{i=1}^n\|x_i\|^p\]
 for all $n\in\N$ and all $x_1,\ldots, x_n\in X$. If $X$ does not have type $p$ for any $p\in (1,2]$, $X$ is said to have \emph{trivial type}. If $q\in [2,\infty)$, we say that $X$ has \emph{cotype $q$} if there is $C>0$ such that 
 \[\frac{1}{2^n}\sum_{(\eps_i)_{i=1}^n\in \{-1,1\}^n}\Big\|\sum_{i=1}^n\eps_ix_i\Big\|^q\geq C\sum_{i=1}^n\|x_i\|^q\]
 for all $n\in\N$ and all $x_1,\ldots, x_n\in X$. We let 
 \[q_X=\inf\{q\in [2,\infty]\mid X \text{ has cotype q}\},\]
 where $q_X$ is taken to be infinity if $X$ has no cotype in $[2,\infty)$. 

Let $p\in [1,\infty]$ and $C\ge 1$.  We say that a Banach space  $X$ \emph{contains the $\ell_p^n$'s $C$-uniformly} if for all $n\in \N$, $\ell_p^n$ linearly embeds into $X$ with distortion at most $C$ and that $X$ \emph{uniformly contains the $\ell_p^n$'s} if it contains the $\ell_p^n$'s $C$-uniformly, for some $C\ge 1$. It is known that if $X$ uniformly contains the $\ell_p^n$'s, then it contains the $\ell_p^n$'s $C$-uniformly for all $C>1$. We shall also use the following fundamental result of B. Maurey and G. Pisier \cite{MaureyPisier1976Studia}: a Banach space $X$ uniformly contains the $\ell_{q_X}^n$'s. We refer the reader to \cite[Theorem 6]{MaureyHandbook} for a precise statement and a self contained proof of this result, as well as for a complete survey on these notions.

\medskip
We are now ready to prove the following statement, which is slightly more precise than Theorem \ref{ThmCotypeACE}.

\begin{theorem}\label{ThmCotype} Let $2\le q<p<\infty$. Suppose that $X$ is a Banach space  uniformly containing  the $\ell_p^n$'s and that $Y$ has cotype $q$ and nontrivial type and suppose also that $0\le \alpha\le\frac{p-q}{p}$. Then there is no coarse map from $X$ to $Y$ that is expanding at rate $\alpha$. 
\end{theorem}

\begin{proof}
Fix a coarse map $f\colon X\to Y$ which is expanding at rate $\alpha$. In particular, there exists  $\rho\colon [0,\infty)\to [0,\infty)$ increasing with $\lim_{t\to \infty}\rho(t)=\infty$ for which 
 \begin{equation}\label{f-embedding}
 \|x-z\|\geq \frac12\max\{\|x\|^\alpha,\|z\|^\alpha\}+1\ \text{ implies }\ \|f(x)-f(z)\|\geq \rho(\|x-z\|).    
 \end{equation} Suppose towards a contradiction that $0<\alpha\le\frac{p-q}{p}$, that is, $p\ge \frac{q}{1-\alpha}$. Treating $\ell_{p}^n(\C)$ as a real Banach space, we have, by assumption, that for each $n\in\N$, there exists an isomorphic embedding $g_n\colon \ell_p^n(\C)\to X$  with distortion at most $2$. Without loss of generality, assume 
 \begin{equation}\label{g_n-embedding}
 \|x\|\leq \|g_n(x)\|\leq 2\|x\|    
 \end{equation}
 for all $n\in\N$ and all  $x\in \ell_p^n(\C)$.  For each $n\in\N$, let $ h_n= f\circ g_n$. It easily follows from \eqref{f-embedding} and \eqref{g_n-embedding} that for $x,z \in \ell_p^n(\C)$,
\[\|x-z\|\geq \max\{\|x\|^\alpha,\|z\|^\alpha\}+1\ \text{ implies }\ \|h_n(x)-h_n(z)\|\geq \rho(\|x-z\|).\] 
Note also that 
\begin{equation}\label{EqThmCotypeACE}
\omega_{h_n}(t)\leq \omega_f(2t)\ \text{ for all } \ n\in\N\ \text{ and all }\ t\geq 0.
\end{equation} 
 It follows from \cite[Theorem 4.1]{MendelNaor2008} that there exists  $\Gamma>0$ such that $m_q(Y,n,\Gamma)=O(n^{1/q})$, i.e.,  there exists $A>0$ such that $m_q(Y,n,\Gamma)\leq An^{1/q}$ for all $n \in \N$ (here is there the hypothesis of $Y$ having nontrivial type in used). On the other hand, by \cite[Lemma 2.3]{MendelNaor2008}, we have that $ m_q(Y,n,\Gamma)\geq \frac{n^{1/q}}{\Gamma}$. We now apply   Lemma \ref{LemmaCotypeInequality} with $q=q$, $r=p$ and $s=\frac1q-\frac1p$. Then $0\le \alpha\le \frac{p-q}{p}=\frac{rs}{rs+1}$. It then follows from   Lemma \ref{LemmaCotypeInequality} and  \eqref{EqThmCotypeACE}  that 
\[\rho\Big(2n^{\frac{1}{q}-\frac{1}{p}}\Big)
\leq \Gamma A \omega_{h_n}\Big(\frac{2\pi n^{1/q}}{m_q(Y,n,\Gamma)}\Big)
\leq\Gamma A \omega_f(4\pi\Gamma)\]
for all $n\in \N$. As $\frac1q-\frac{1}{p}>0$ and $\lim_{t\to \infty}\rho(t)=\infty$, we obtain the expected contradiction.  
 \end{proof}

\begin{corollary}\label{CorEmblpIntolq}
Let $1\le q<p<\infty$ and $\alpha\in [0,1]$.
\begin{enumerate}
    \item\label{CorEmblpIntolqItem1} If $q\geq 2$, then  there is a coarse map $f\colon \ell_p\to \ell_q$ which is also expanding at rate  $\alpha$ if and only if $\alpha>\frac{p-q}{p}$.
    \item\label{CorEmblpIntolqItem2} If $q\leq 2<p$, then  there is a coarse map $f\colon \ell_p\to \ell_q$ which is also expanding at rate  $\alpha$ if and only if $\alpha>\frac{p-2}{p}$.
\end{enumerate} 
\end{corollary}

\begin{proof}
\eqref{CorEmblpIntolqItem1} Assume first that $\alpha \in (\frac{p-q}{p},1]$. The existence of a coarse map $f\colon \ell_p\to \ell_q$ which is  expanding at rate  $\alpha$ is ensured by Theorem  \ref{ThmMazurMapACE}.

Assume now that $\alpha \in [0,\frac{p-q}{p}]$. Since $\ell_p$ obviously contains the $\ell_p^n$'s uniformly and $\ell_q$ has cotype $q$, Theorem \ref{ThmCotype} implies that there is no coarse map from $\ell_p$ to $\ell_q$ that is expanding at rate $\alpha$.

\eqref{CorEmblpIntolqItem2} Since $q\leq 2$, $\ell_q$ has cotype 2. Hence, proceeding as in the previous item, we have that  there is no coarse map $f\colon \ell_p\to \ell_q$ which is also expanding at rate  $\alpha$ for $\alpha\in [0,\frac{p-2}{p}]$. If $\alpha\geq  \frac{p-q}{p}$, the Mazur map $M_{p,q}\colon \ell_p\to \ell_q$   is also expanding at rate  $\alpha$ (Theorem \ref{ThmMazurMapACE}).  Finally, if $q\in (\frac{p-2}{p},\frac{p-q}{p}]$,  the result follows since the Mazur map $M_{p,2}\colon \ell_p\to \ell_2$   is   expanding at rate  $\alpha$. Therefore, if $f\colon \ell_2\to \ell_q$ is a coarse embedding (see \cite[Theorem 5]{Nowak2006FundMath}), then $f\circ M _{p,2}$ has the required property. 
\end{proof}

\begin{proof}[Proof of Theorem \ref{ThmCotypeACE}]
If $\alpha=0$, this is  \cite[Theorem 1.11]{MendelNaor2008} and if $\alpha=1$, this is obvious. Assume now that $\alpha \in (0,1)$. We may also assume that $q_Y<\infty$. As we have recalled, by \cite{MaureyPisier1976Studia} (alternatively, see \cite[Theorem 6]{MaureyHandbook}), $X$ uniformly contains the $\ell_{q_X}^n$'s. On the other hand, for any $q>q_Y$, $Y$ has cotype $q$. It then follows from Theorem \ref{ThmCotype} that $\alpha >\frac{q_X-q}{q_X}$. Taking the infimum over all $q>q_Y$ yields the conclusion. 
\end{proof}

\section{Embedding of certain graphs and expansion}

In this section we will show how to use some classical countably branching metric graphs to rule out the existence of   coarse maps which are linearly expanding (or just expanding) at rate $1$ or or less  between Banach spaces. Our statements will actually be about the stability of some concentration properties for Lipschitz maps on certain graphs under these generalizations of coarse Lipschitz (or just coarse) embeddings. We will deal with two different distances that are defined on the following sets. 

\begin{definition}
Let $k\in\N$. We denote the set  of all subsets of $\N$ with $k$ elements by $[\N]^k$ and, given $\bar n\in[\N]^k$, we   write $\bar n=(n_1,\ldots, n_k)$ where  $n_1<\ldots< n_k$.
\end{definition}

\subsection{Hamming graphs}
We start this section by recalling the definition of the Hamming metric and of the associated concentration property HFC$_p$. 

\begin{definition}
Let $k\in\N$.  The  \emph{Hamming metric} $d_{\mathbb H}=d_{\mathbb H,k}$ on $[\N]^k$ is defined as follows: given $\bar n=(n_1,\ldots, n_k),\bar m=(m_1,\ldots, m_k)\in [\N]^k$, we let 
\[d_{\mathbb H}(\bar n, \bar m)=|\{i\in \{1,\ldots, k\}\mid n_i\neq m_i\}|.\]
\end{definition}

Note that $([\N]^k,d_{\mathbb H})$ is a metric graph in the sense that if two elements of $[\N]^k$ are declared to be adjacent if and only if they are at distance $1$, then the distance between arbitrary elements of $[\N]^k$ is the length of the shortest path joining them. In particular, if $X$ is a Banach space and $\phi:([\N]^k,d_{\mathbb H}) \to X$ is a Lipschitz map, then the Lipschiz constant of $\phi$ is $\Lip(\phi)=\omega_\phi(1)$. We now recall the following definition, due to A. Fovelle (see \cite[Subsection 2.5]{Fovelle}). 

 \begin{definition}[Definition \ref{DefHCF}] Let $p\in (1,\infty]$. We say that a Banach space $X$ has the \emph{Hamming full concentration property for $p$} (abbreviated as \emph{HFC$_p$}) if there exists a constant $C\ge 1$ such that for any $k\in \N$ and any $1$-Lipschitz map $\phi\colon ([\N]^k,d_{\mathbb H}) \to X$, there exists an infinite subset $\M$ of $\N$ such that \[\diam(\phi([\N]^k))\le Ck^{1/p}.\] 
 In the above situation, we say that $X$ has the \emph{HFC$_p$ with constant $C$}. 
\end{definition}

\begin{example}\label{example HFC}
N. Kalton and L. Randrianarivony proved in \cite[Theorem 4.2]{KaltonRandrianarivony2008} that, given $p\in (1,\infty)$,  any reflexive Banach space which is also $p$-asymptotically uniformly smooth (abbreviated as $p$-AUS) must have  HFC$_p$. Although this will not play an important role in these notes, we recall the definition of $p$-AUSness here for the readers convenience: the \emph{modulus of asymptotic uniform smoothness} of a Banach space $X$ is given by 
  \[\bar\rho_X(t)=\sup_{x\in \partial B_X}\inf_{Y\in \mathrm{cof}(X)}\sup_{y\in \partial B_Y}\|x+ty\|-1,\]
where $\mathrm{cof}(X)$ denotes the set of all closed finite codimensional subspaces of $X$. Then, for $p\in (1,\infty)$, $X$ is called \emph{$p$-AUS} if there is $C\geq 1$ such that $\bar \rho_X(t)\leq Ct^p$ for all $t\geq 0$.   Clearly, for $p\in (1,\infty)$, $\ell_p$ is $p$-AUS, hence, the cited result above implies that $\ell_p$ has HFC$_p$. But, if $q<p$, $\ell_q$ fails HFC$_p$ as it is witnessed by the maps   
  \[\bar n=(n_1,\ldots,n_k)\in [\N]^k\mapsto \sum_{i=1}^ke_i\in \ell_q, \ k\in\N,\]
  where $(e_i)_{i=1}^\infty $ is the canonical unit basis of $\ell_q$. 
\end{example}

\begin{example}\label{ExaAsympc_0}
A Banach space $X$ is said to be \emph{asymptotic-$c_0$} if the following happens: 
\begin{gather*}
    \exists C\geq 1,\ \forall n\in\N,\ \exists X_1\in \mathrm{cof}(X), \forall x_1\in B_{X_1},\ \ldots,\ \exists X_n\in \mathrm{cof}(X), \forall x_n\in B_{X_n}\\
    \text{ such that }  \   \Big\|\sum_{i-1}^na_ix_i\Big\|\leq C \max_{1\leq i\leq n}|a_i|\ \text{ for all } \ (a_i)_{i=1}^n\in \R^\N.
    \end{gather*}
 The second named author together with  F.  Baudier, P.  Motakis, and Th. Schlumprecht proved that the property of a Banach space having HFC$_\infty$ is equivalent to it being reflexive and asymptotic-$c_0$ (see   \cite[Theorem B]{BaudierLancienMotakisSchlumprecht2018}).
\end{example}

\begin{proof}[Proof of Theorem \ref{ThmPreservationHFCpByLinearExpansion}\eqref{ThmPreservationHFCpByLinearExpansionItem1}]  Let $f\colon X\to Y$ be a coarse map which is linearly expanding at rate $1$. We start by fixing  some constants: As $f$ is coarse, $\omega_f(1)<\infty$ and it follows that there is  $K>1$ such that \[\omega_f(t)\le Kt+K, \text{ for all } t\ge 0\] 
(this easy consequence of the triangle inequality and the metric convexity of Banach spaces can be found, for instance, in \cite[Lemma 1.4]{KaltonSurvey}). As $f$ is linearly expanding at rate $1$, there is $L>1$ such that 
\begin{equation}\label{Eq1ThmPreservationHFCpByLinearExpansion}\|x-z\|\geq \frac{1}{8}\max\{\|x\|,\|z\|\}+1\ \text{ implies  }\ \|f(x)-f(z)\|\geq \frac{1}{L}\|x-z\|-L.
\end{equation}
Let  $C_Y\ge 1$ be a constant such that   $Y$ has HFC$_p$ with constant $C_Y$ and fix  $C>6KLC_Y+6L^2$.

Assume towards a contradiction that $X$ does not have HFC$_p$. Then there exist $k\in \N$ and a $1$-Lipschitz map $\phi\colon ([\N]^k,d_{\mathbb H})\to X$  such that  \[\diam(\phi([\M]^k))\ge Ck^{1/p}\]
for all infinite $\M\subseteq \N$. Let 
\begin{equation}\label{eqlambda}
\lambda=\inf\{k^{-1/p}\diam(\phi([\M]^k))\mid \M\subseteq\N \text{ and }|\M|=\infty\};
\end{equation}  so, $\lambda\ge C$ (in particular, $\lambda>12$). Pick now an infinite $\M\subseteq \N$ such that 
\[\lambda \le \diam(\phi([\M]^k))k^{-1/p} \le 2\lambda.\]    In particular, there exists $x\in X$ such that $\phi([\M]^k)$ is included in the closed ball of radius $2\lambda k^{1/p}$ centered at $x$. By replacing $\phi$ by $\phi-x$,  we may assume that $\phi([\M]^k)$ is included in $2\lambda k^{1/p} B_X$. 

By the definition of $\lambda$,  $\diam(\phi([\D]^k))\ge \lambda k^{1/p}$  for all infinite  $\D\subseteq \M$.  So, for any such $\D$, we can find $\bar{n},\bar{m}\in [\D]^k$ with $\bar n< \bar m$   such that $\|\phi(\bar{n})-\phi(\bar{m})\|\ge \frac{\lambda}{3}k^{1/p}$. Indeed,  for all $\bar{n},\bar{m}\in [\D]^k$, we can find $\bar{p}\in [\D]^k$ such that $\bar{n}<\bar{p}$ and $\bar{m}<\bar{p}$. Therefore, if $\|\phi(\bar{n})-\phi(\bar{m})\|$ were smaller than $\frac{\lambda}{3}k^{1/p}$ for all $\bar n<\bar m$ in $[\D]^k$, the triangle inequality would imply that $\diam(\phi([\D]^k))<  \lambda k^{1/p}$; contradiction. Moreover, identifying a pair $(\bar n,\bar m)$, where $\bar n,\bar m\in [\D]^k$ and $\bar{n}<\bar{m}$, with an element of $[\M]^{2k}$ and applying Ramsey's theorem, we may furthermore assume, by passing to an infinite subset of $\M$  if necessary, that \begin{equation}\label{EqHFCp}
\|\phi(\bar{n})-\phi(\bar{m})\|\ge \frac{\lambda}{3}k^{1/p}\text{  for all }\bar n,\bar m\in [\M]^k\text{ with }\bar{n}<\bar{m}. 
\end{equation}
As $\phi([\M]^k)$ is included in $2\lambda k^{1/p} B_X$, \eqref{EqHFCp}  and the fact that $\lambda>12$ imply that   
\begin{align*}
\frac{1}{8}\max\Big\{\|\phi(\bar{n})\|,{\|\phi(\bar{m})\|}\Big\}+1 \le \frac{\lambda k^{1/p}}{4}+1 \le \frac{\lambda k^{1/p}}{3} \le  \|\phi(\bar{n})-\phi(\bar{m})\| 
\end{align*}
for all $\bar n,\bar m\in [\M]^k$ with $\bar{n}<\bar{m}$. Therefore, by \eqref{Eq1ThmPreservationHFCpByLinearExpansion},  we must have
\begin{equation}\label{Eq.HFCp2}
\|(f\circ\phi)(\bar{n})-(f\circ\phi)(\bar{m})\|\ge  \frac{\lambda k^{1/p}}{3L}-L \ge \frac{C k^{1/p}}{3L}-L
\end{equation}
 for all $\bar n,\bar m\in [\M]^k$ with $\bar{n}<\bar{m}$. 

On the other hand, as   $\phi$ is 1-Lipschitz and $[\M]^k$ is a metric graph, our choice of $K$   implies that $\omega_{f\circ \phi}(1)\le 2K$ and therefore that $f\circ \phi$ is $2K$-Lipschitz. As $Y$ has HFC$_p$ with constant $C_Y$, an homogeneity argument implies the existence of an infinite subset $\M'$ of $\M$ satisfying 
\begin{equation}\label{Eq.HFCp3}
\diam(f(\phi([\M']^k)))\le 2KC_Y k^{1/p}.
\end{equation}
Therefore, \eqref{Eq.HFCp2} and \eqref{Eq.HFCp3} imply that
\begin{equation}\label{Eq.HFCp4}\frac{C}{3L}\leq 2KC_Y+L.
\end{equation}
This contradicts our choice of $C$.
\end{proof}

Combining this result and Example \ref{example HFC}, we obtain immediately the following.

\begin{corollary} Let $p,q \in [1,\infty)$ and assume that $p<q$. Then there is no coarse map from $\ell_p$ to $\ell_q$ that is linearly expanding at rate $1$. 
\end{corollary}

\begin{remark} It is important to recall here that if $p>q$, then the Mazur map $M_{p,q}:\ell_p \to \ell_q$ is coarse and linearly expanding at rate $1$ (Theorem \ref{ThmMazurMapACE}). 
\end{remark}

 In case $p=\infty$, a stronger version of Theorem \ref{ThmPreservationHFCpByLinearExpansion} holds. Precisely,  HFC$_\infty$ is preserved by coarse maps which are expanding at rate $1$; no need for linear expansion here. Since the argument is completely analogous, we only indicate the mild differences in the proof below.

 \begin{proof}[Proof of Theorem \ref{ThmPreservationHFCpByLinearExpansion}\eqref{ThmPreservationHFCpByLinearExpansionItem2}] 
The proof follows almost  verbatim  the one of  Theorem \ref{ThmPreservationHFCpByLinearExpansion}\eqref{ThmPreservationHFCpByLinearExpansionItem1}. Precisely, let $f\colon X\to Y$ be a coarse map which is expanding at rate $1$  and $K$ and $C_Y$ be chosen as in the proof of Theorem \ref{ThmPreservationHFCpByLinearExpansion}\eqref{ThmPreservationHFCpByLinearExpansionItem1}. The expansion property of $f$ gives us   $\rho\colon [0,\infty)\to [0,\infty)$ with  $\lim_{t\to \infty}\rho(t)=\infty$ and such that \[\|x-z\|\geq \frac{1}{8}\max\{\|x\|,\|z\|\}+1 \ \text{ implies  }\ \|f(x)-f(z)\|\geq \rho(\|x-z\|).\] 

Since $\lim_{t\to \infty}\rho(t)=\infty$, we can  choose $C>0$ such that $\rho(C /3)>2KC_Y$ and, assuming that $X$ does not have HFC$_\infty$, we obtain $k\in\N$ and a $1$-Lipschitz map $\phi\colon ([\N]^k,d_{\mathbb H})\to X$ such that 
\[\mathrm{diam}(\phi([\M]^k))\geq C\]
for all infinite $\M\subseteq \N$. 
Defining  $\lambda$ as in \eqref{eqlambda} (here $1/\infty=0$) and letting $\M\subseteq \N$ be as in the proof of  Theorem \ref{ThmPreservationHFCpByLinearExpansion}, the proof then proceeds verbatim  until \eqref{Eq.HFCp2} which in this case is replaced  by 
\begin{equation}\label{InftyEq.HFCp2}
\|(f\circ\phi)(\bar{n})-(f\circ\phi)(\bar{m})\|\ge  \rho\Big(\frac{\lambda}{3}\Big) > 2KC_Y
\end{equation}
 for all $\bar n,\bar m\in [\M]^k$ with $\bar{n}<\bar{m}$. Then, \eqref{Eq.HFCp3} becomes
\[
\diam(f(\phi([\M']^k)))\le 2KC_Y.
\] and both those inequalities put together give us a contradiction.
 \end{proof}

\begin{proof}[Proof of Corollary \ref{CorollaryHFCInfty}]
 By  \cite[Theorem B]{BaudierLancienMotakisSchlumprecht2018}, we know that having HFC$_\infty$ is equivalent to being reflexive and asymptotic-$c_0$. The result is then an immediate consequence of     Theorem \ref{ThmPreservationHFCpByLinearExpansion}\eqref{ThmPreservationHFCpByLinearExpansionItem2}.
\end{proof}
\subsection{Interlacing pairs in Hamming graphs}\label{SubsectionInterlacingHamming}

Property HFC$_p$, for $p\in (1,\infty]$, implies reflexivity (this follows from \cite[Theorem 4.1]{BaudierKaltonLancien2010Studia}). Therefore this property is not relevant to study embeddings between non reflexive Banach spaces. With the goal of addressing this problem, the following weakening  of the HFC$_p$ was  introduced in \cite{LancienRaja2017Houston} and formalized in \cite[Subsection 2.5]{Fovelle}.

\begin{definition} Let $\M$ be an infinite subset of $\N$ and let $k\in \N$. The   set of \emph{strictly interlacing} pairs in $[\M]^k$ is given by 
$$I_k(\M)=\{(\bar n,\bar m)\in [\M]^k\times [\M]^k\mid \ n_1< m_1< n_2< m_2< \cdots < n_k < m_k\}.$$
Given a Banach space $X$ and  $p\in (1,\infty]$, we say that   $X$ has the \emph{Hamming interlaced concentration property for $p$}, abbreviated as \emph{HIC}$_p$, if there exists a constant $C\ge 1$ such that for any $k\in \N$ and any $1$-Lipschitz map $\phi\colon ([\N]^k,d_{\mathbb H}) \to X$, there exists an infinite subset $\M$ of $\N$ such that 
$$\|f(\bar n)-f(\bar m)\|\le Ck^{1/p},\ \text{for all}\ (\bar n,\bar m) \in I_k(\M).$$
\end{definition}

\begin{example}\label{ExampleJamesSpace} Recall, given $p\in (1,\infty)$, that  the James sequence space $J_p$ is defined by 
\begin{align*}
J_p=\Big\{(x(n))_n\in \R^\N\mid & \lim_nx(n)=0\ \text{ and }\\ &\|x\|_{J_p}=\sup_{p_1<\ldots<p_n}\Big(\sum_{i=2}^n|x(p_{i})-x(p_{i-1})|^p\Big)^{1/p}<\infty\Big\} .
\end{align*}
So, the classic James   sequence space $J$ is simply $J_2$ and, just as $J$, each $J_p$ has codimension $1$ in its bidual; which makes them  quasi-reflexive Banach spaces. M. Raja and the second named author proved in \cite[Theorem 2.2]{LancienRaja2017Houston} that, for $p\in (1,\infty)$, James space $J_p$ has HIC$_p$ and that, for $q<p$, $J_q$ fails HIC$_p$. 
\end{example}

We can prove the following variant of Theorem \ref{ThmPreservationHFCpByLinearExpansion}. Notice that the assumption on the rate of expansion has been weakened. 

\begin{theorem}\label{ThmInterlacingHamming} Let $X$ and $Y$ be Banach spaces and suppose $Y$ has HIC$_p$ for some $p\in (1,\infty)$. If there is a coarse map $f\colon X\to Y$ which is linearly expanding at rate $\frac1p$, then  $X$ must have HIC$_p$.
\end{theorem}
 
\begin{proof} The proof is  similar to the one of Theorem \ref{ThmPreservationHFCpByLinearExpansion}, but we need to detail some of the minor modifications. So, let $f\colon X\to Y$ be as in the statement  and $C_Y\ge 1$ be such that $Y$ has HIC$_p$ with constant $C_Y$. There exists $K\ge 1$ so that $\omega_f(t)\le Kt+K$ and $L\ge 1$ such that 
\begin{equation}\label{Eq2ThmPreservationHFCpByLinearExpansion}\|x-z\|\geq \max\{\|x\|^{1/p},\|z\|^{1/p}\}+1\ \text{ implies  }\ \|f(x)-f(z)\|\geq \frac{1}{L}\|x-z\|-L.
\end{equation}
Assume that $X$ fails HIC$_p$ and fix $C>3$ (to be precised later). Then there exist $k\in \N$ and a $1$-Lipschitz map $\phi\colon ([\N]^k,d_{\mathbb H})\to X$  such that  for any infinite $\M\subseteq \N$ there exists $(\bar n,\bar m) \in I_k(\M)$ so that $\|\phi(\bar n)-\phi(\bar m)\|> Ck^{1/p}$. Since $\Phi$ is 1-Lipschitz and $\diam([\N]^k)=k$, we may assume that $\phi([\N]^k)\subseteq kB_X$. Let 
\begin{equation}\label{eqlambdabis}
\lambda=\inf\Big\{ \sup_{(\bar n,\bar m)\in I_k(\M)}\|\phi(\bar n)-\phi(\bar m)\|k^{-1/p}\mid\ \M\subseteq\N \text{ and }|\M|=\infty\Big\}.
\end{equation}  
So, $\lambda\ge C$. Pick now an infinite $\M\subseteq \N$ such that 
\[\lambda \le \sup_{(\bar n,\bar m)\in I_k(\M)}\|\phi(\bar n)-\phi(\bar m)\|k^{-1/p} < 2\lambda.\] 
By the definition of $\lambda$, for all infinite $\D\subseteq \M$, there exists $(\bar n,\bar m)\in I_k(\D)$ such that $\|\phi(\bar n)-\phi(\bar m)\|> \frac23 \lambda k^{1/p}$. Identifying $I_k(\D)$ with $[\D]^{2k}$ and using Ramsey's theorem,  we can therefore assume that $\|\phi(\bar n)-\phi(\bar m)\|> \frac23 \lambda k^{1/p}$ for all $(\bar n,\bar m) \in I_k(\M)$. It follows that for all $(\bar n,\bar m) \in I_k(\M)$:
\begin{align*}
\|\phi(\bar n)-\phi(\bar m)\|\ge  \frac{2C}{3} k^{1/p} \ge \frac{C}{3} k^{1/p}+1 \ge \max\{\|\phi(\bar n)\|^{1/p},\|\phi(\bar m)\|^{1/p}\}+1.
\end{align*}
It then follows from (\ref{Eq2ThmPreservationHFCpByLinearExpansion}) that 
$$\|(f \circ \phi)(\bar n)-(f \circ \phi)(\bar m)\|\ge  \frac{2C}{3L}k^{1/p}-L,$$
for all $(\bar n,\bar m) \in I_k(\M)$.

On the other hand $f \circ \phi$ is $2K$-Lipschitz and $Y$ has property HIC$_p$ with constant $C_Y$, so there exists an infinite subset $\M'$ of $\M$ so that $\|(f \circ \phi)(\bar n)-(f \circ \phi)(\bar m)\|\le 2KC_Yk^{1/p}$, for all  $(\bar n,\bar m) \in I_k(\M')$. This yields a contradiction for an initial large enough choice of  $C$ (depending on $K,L,C_Y$).
\end{proof}

\begin{remark} In the case $p=\infty$, the analogous statement is that HIC$_\infty$ is preserved by coarse maps expanding at rate $0$, in other words by coarse embeddings. This was already noticed in \cite{Fovelle}. 
\end{remark}

\begin{problem}\label{ProblemHICpRate1}
Let $p\in (1,\infty]$. We do not know if the property HIC$_p$ is preserved by coarse maps that are linearly expanding at rate 1.
\end{problem}

As an immediate application of Theorem \ref{ThmInterlacingHamming} and Example \ref{ExampleJamesSpace}, we get the following. 

\begin{corollary}\label{CorInterlacingHammingJames} Let $p,q \in [1,\infty)$ and assume that $p<q$. Then there is no coarse map from $J_p$ to $J_q$ that is linearly expanding at rate $1/q$. \qed
\end{corollary}

\subsection{Interlacing graphs}
We now deal with the interlacing metric of $[\N]^k$, which was introduced by N. Kalton in \cite{Kalton2007} to rule out the coarse embeddability of $c_0$ into reflexive spaces. We start by recalling the definition of the interlaced metric and of property $\mathcal Q_p$. 

\begin{definition}
Let $k\in\N$.  The  \emph{interlacing  metric} $d_{\mathbb I}=d_{\mathbb I,k}$ on $[\N]^k$ is defined as follows: we endow $[\N]^k$ with a graph structure by letting  distinct $\bar n=(n_1,\ldots, n_k),\bar m=(m_1,\ldots, m_k)\in [\N]^k$ be adjacent if either 
\[n_1\le m_1\le \ldots \le n_k\le m_k\ \text{ or }\ m_1\le n_1\le \ldots\le m_k\le n_k.\]
The metric $d_{\mathbb I}$ is then the shortest path metric on $[\N]^k$ with respect to this graph structure. 
\end{definition}

 \begin{definition}[Definition \ref{DefiPropQ}] Let $p\in (1,\infty]$. We say that a Banach space $X$ has property $\mathcal Q_p$ if there exists a constant $C\ge 1$ such that for any $k\in \N$ and any $1$-Lipschitz map $\phi\colon ([\N]^k,d_{\mathbb I}) \to X$, there exists an infinite subset $\M$ of $\N$ such that \[\diam(\phi([\N]^k))\le Ck^{1/p}.\] 
(here if $p=\infty$, we use the convention $1/\infty=0)$. If $p=\infty$, we simply say, following \cite{Kalton2007}, that $X$ has \emph{property $\mathcal Q$}.
\end{definition}

\begin{example} N. Kalton proved in \cite[Corollary 4.3]{Kalton2007} that reflexive Banach spaces have property $\mathcal Q$. In \cite{BragaLancienPetitjeanProchazka2023JTopAna}, it is shown that, for $p\in (1,\infty)$, a dual of a $p$-AUS Banach space has $\mathcal Q_{p'}$, where $p'$ is the conjugate of $p$ and that the dual of an asymptotic-$c_0$ space has $\mathcal Q$ (\cite[Theorem 4.1]{BragaLancienPetitjeanProchazka2023JTopAna}). 
\end{example}

We now  extend the result  insuring the stability of property $\mathcal Q_p$ under coarse Lipschitz embeddings.

 \begin{proof}[Proof of Theorem \ref{ThmPreservationPropertyQpByLinearExpansion}\eqref{ThmPreservationPropertyQpByLinearExpansionItem1}] Again the proof follows the lines of the argument for Theorem \ref{ThmPreservationHFCpByLinearExpansion}\eqref{ThmPreservationHFCpByLinearExpansionItem1}. let $f\colon X\to Y$ be as in the statement and $C_Y$ be such that $Y$ has $\mathcal Q_p$ with constant $C_Y$. Then the constants $K,L,C$ are defined as in the proof of Theorem \ref{ThmPreservationHFCpByLinearExpansion}\eqref{ThmPreservationHFCpByLinearExpansionItem1} and assume that $X$ fails $\mathcal Q_p$. Then there exist $k\in \N$ and a $1$-Lipschitz map $\phi\colon ([\N]^k,d_{\mathbb I})\to X$  such that  for any infinite $\M\subseteq \N$, 
  \[\mathrm{diam}(\phi([\M]^k))\geq C\] 
Let 
\begin{equation}
\lambda=\inf\{\diam(\phi([\M]^k))k^{-1/p}\mid \M\subseteq\N \text{ and }|\M|=\infty\};
\end{equation}  and we pick an infinite $\M\subseteq \N$ such that 
\[\lambda \le \diam(\phi([\M]^k))k^{-1/p} < 2\lambda.\] So assume, as we may, that $\phi([\M]^k) \subseteq 2\lambda k^{1/p}B_X$. By the definition of $\lambda$ and arguing as in the proof of Theorem \ref{ThmPreservationHFCpByLinearExpansion}\eqref{ThmPreservationHFCpByLinearExpansionItem1}, we may assume that for all $\bar n < \bar m \in [\M]^k$, $\|\phi(\bar{n})-\phi(\bar{m})\|\ge \frac{\lambda}{3}k^{1/p}$. As $\phi([\M]^k)$ is included in $2\lambda k^{1/p}\cdot B_X$, we deduce similarly that 
\begin{equation}
\|(f\circ\phi)(\bar{n})-(f\circ\phi)(\bar{m})\|\ge  \frac{\lambda k^{1/p}}{3L}-L \ge \frac{C k^{1/p}}{3L}-L
\end{equation}
 for all $\bar n,\bar m\in [\M]^k$ with $\bar{n}<\bar{m}$. The contradiction then follows exactly as in the proof of Theorem \ref{ThmPreservationHFCpByLinearExpansion}. 
\end{proof}

\begin{proof}[Proof of Theorem \ref{ThmPreservationPropertyQpByLinearExpansion}\eqref{ThmPreservationPropertyQpByLinearExpansionItem2}]
We just have to adapt similarly the proof of Theorem \ref{ThmPreservationHFCpByLinearExpansion}\eqref{ThmPreservationHFCpByLinearExpansionItem2}. We leave the details to the reader.
\end{proof}

Asymptotic uniform convexity is often used, together with the approximate midpoint principle to find obstructions to the coarse Lipschitz embeddability. This is how it is shown that $\ell_p$ does not coarse Lipschitz embed into $\ell_q$ for $p>q$ (see \cite{JohnsonLindenstraussSchechtman1996GAFA}).  We cannot find an approximate midpoint principle for coarse maps that are linearly expanding at rate $1$ and, as we already emphasized, there is a good reason for that: it follows from  Theorem  \ref{ThmMazurMapACE} that $M_{p,q}:\ell_p \to \ell_q$ is a coarse map from that is linearly expanding at rate $1$ when $p>q$. However, the situation is different for the family of James spaces $(J_p)_p$ and the use of property $Q$ and its variants can serve as a substitute to obtain some preservation of asymptotic uniform convexity. 

\begin{corollary} Let $p,q \in (1,\infty)$ and assume that $p>q$. Then there is no coarse map from $J_p$ to $J_q$ that is linearly expanding at rate $1$.
\end{corollary}

\begin{proof} It is proved by the authors, C. Petitjean and A. Proch\'azka in \cite[Corollaries 3.4 and 3.8]{BragaLancienPetitjeanProchazka2023JTopAna}  that, for $p\in (1,\infty)$, $J_p$ has property $\mathcal Q_{p'}$, where $p'$ is the conjugate exponent of $p$, but fails $\mathcal Q_{r}$ for all $r>p$. Then the conclusion follows from Theorem \ref{ThmPreservationPropertyQpByLinearExpansion}\eqref{ThmPreservationPropertyQpByLinearExpansionItem2}.
\end{proof}


\section{Embeddings into $\ell_\infty$}\label{SectionlInfty}
 In this last section, we prove that Lipschitz embeddability into $\ell_\infty $ is equivalent to the existence of a coarse map which is linearly expanding at some rate strictly smaller than $1$. We start with an intermediate result.  
 
\begin{proposition}\label{PropositionlInfty}
Let $\alpha\in (0,1)$. Let $X$ be a Banach space and suppose there is a Lipschitz map $f\colon X\to \ell_\infty$ which is also  linearly  expanding at rate  $\alpha$. Then $X$ Lipschitzly embeds into $\ell_\infty$. 
\end{proposition} 

\begin{proof}
Fix $L>0$ such that, for all $x,z \in X$,  
 \[\|x-z\|\geq L(\max\{\|x\| ,\|z\|\})^\alpha +L\ \text{ implies }\  \|f(x)-f(z)\|\geq \frac{1}{L}\|x-z\|-L.\]
Denote $\Q_+$ the set of positive rational numbers and define a map $F\colon X\to \ell_\infty(\Q_+,\ell_\infty)$ by letting \[F(x)=(q^{-1}f(qx))_{q\in  \Q_{+}},\ \ x\in X.\]  
Since, 
\[\|F(x)-F(z)\|=\sup_{q\in \Q_{+}}q^{-1}\|f(qx)-f(qz)\|\leq \Lip(f)\|x-z\|,\]
for all $x,z \in X$, we have that $\Lip(F)\leq \Lip(f)$. 

\noindent Fix now $x,z\in X$ with $x\neq z$. As  $\alpha\in (0,1)$, there is $t>0$ large enough so that 
\[\Big\|\frac{tx}{\|x-z\|}-\frac{tz}{\|x-z\|}\Big\|=t\geq L \Big( \max\Big\{\frac{\|tx\|}{\|x-z\|},\frac{\|tz\|}{\|x-z\|}\Big\}\Big)^\alpha+L.\] Taking an even larger $t$ if necessary, we can also assume that $L<\frac{t}{2L}$. We may also assume that $q=\frac{t}{\|x-z\|} \in \Q_+$. We obtain that 
\begin{align*}
\|F(x)-F(z)\|&\geq \frac{\|x-z\|}{t}\Big\|f\Big(\frac{tx}{\|x-z\|}\Big)-f\Big(\frac{tz}{\|x-z\|}\Big)\Big\|\\ 
&\geq \frac{\|x-z\|}{t}\Big(\frac{t}{L}-L\Big)\geq \frac{1}{2L}\|x-z\|   
\end{align*}
As $x$ and $z$ were arbitrary, this shows that $F$ is a Lipschitz embedding from $X$ into  $\ell_\infty(\Q_+,\ell_\infty)$, which is clearly isometric to $\ell_\infty$.
\end{proof}

 \begin{proof}[Proof of Theorem \ref{ThmlInfty}]
 The implication \eqref{ThmlInftyItem1}$\Rightarrow$\eqref{ThmlInftyItem2} is immediate. Suppose \eqref{ThmlInftyItem2} holds and let $f\colon X\to \ell_\infty$ be such a map.  Let $N\subseteq X$ be a net of $X$, i.e., for some $\delta,\eps>0$ the set $N$ is $\delta$-separated and $\eps$-dense in $X$. Since $f$ is coarse and $X$ is metrically convex, as we have already seen (cf \cite[Lemma 1.4]{KaltonSurvey}), $f$ is coarse-Lipschitz, in fact, we have that 
 \[\|f(x)-f(z)\|\leq \omega_f(1)\|x-z\|+\omega_f(1)\]
 for all $x,z\in X$. Therefore, as $N$ is $\delta$-separated, $f\restriction N$, the restriction of $f$ to $N$, is Lipschitz with $\Lip(f\restriction N)\leq \omega_f(1)(1+1/\delta)=C$.  Since any Lipschitz map into $\ell_\infty$ can be extended to larger subsets without increasing its Lipschitz constant (see \cite[Section 3.3]{KaltonSurvey}), there is a Lipschitz map $h\colon X\to \ell_\infty$ such that $h\restriction N=f\restriction N$ and $\Lip(h)=C$. As $N$ is $\eps$-dense in $X$ and $f$ and $h$ coincide on $N$, we easily deduce that $\|f-h\|\le C\eps + \omega_f(\eps)$ on $X$. It then follows that $h$ is also linearly expanding at rate  $\alpha$. Then, by  Proposition \ref{PropositionlInfty},  there exists a Lipschitz embedding from $X$ into $\ell_\infty$ and \eqref{ThmlInftyItem1} follows. 
 \end{proof}
 
 \begin{problem}
 We do not know whether we can take $\alpha=1$ in the statement of Theorem \ref{ThmlInfty}. 
 \end{problem}

 \begin{problem} It is proved in \cite[Theorem 5.3]{Kalton2011FundMath} that the Lipschitz embeddability of a Banach space $X$ into $\ell_\infty$ is in fact equivalent to its coarse embeddability. We do not know whether it is also equivalent to the existence of coarse map from $X$ to $\ell_\infty$ that is expanding at some nontrivial rate $\alpha \in (0,1]$; notice that  Theorem \ref{ThmlInfty} assumes \emph{linear} expansion at some rate $\alpha\in (0,1)$.   
 \end{problem}


\begin{thebibliography}{BLMS21}

\bibitem[BKL10]{BaudierKaltonLancien2010Studia}
F.~Baudier, N.~Kalton, and G.~Lancien.
\newblock A new metric invariant for {B}anach spaces.
\newblock {\em Studia Math.}, 199(1):73--94, 2010.

\bibitem[BL00]{BenyaminiLindenstraussBook}
Y.~Benyamini and J.~Lindenstrauss.
\newblock {\em Geometric nonlinear functional analysis. {V}ol. 1}, volume~48 of
  {\em American Mathematical Society Colloquium Publications}.
\newblock American Mathematical Society, Providence, RI, 2000.

\bibitem[BL23]{BragaLancien2023Equiv}
B.~M. {Braga} and G.~{Lancien}.
\newblock {Asymptotic coarse Lipschitz equivalence}.
\newblock {\em arXiv e-prints}, page arXiv:2302.12016, February 2023.

\bibitem[BLMS21]{BaudierLancienMotakisSchlumprecht2018}
F.~Baudier, G.~Lancien, P.~Motakis, and Th. Schlumprecht.
\newblock A new coarsely rigid class of {B}anach spaces.
\newblock {\em J. Inst. Math. Jussieu}, 20(5):1729--1747, 2021.

\bibitem[BLPP23]{BragaLancienPetitjeanProchazka2023JTopAna}
B.~M. Braga, G.~Lancien, C.~Petitjean, and A.~Proch\'{a}zka.
\newblock On {K}alton's interlaced graphs and nonlinear embeddings into dual
  {B}anach spaces.
\newblock {\em J. Topol. Anal.}, 15(2):467--494, 2023.

\bibitem[BLS18]{BaudierLancienSchlumprecht2018}
F.~Baudier, G.~Lancien, and Th. Schlumprecht.
\newblock The coarse geometry of {T}sirelson's space and applications.
\newblock {\em J. Amer. Math. Soc.}, 31(3):699--717, 2018.

\bibitem[Bra17]{Braga2017JFA}
B.~M. Braga.
\newblock Coarse and uniform embeddings.
\newblock {\em J. Funct. Anal.}, 272(5):1852--1875, 2017.

\bibitem[Fov22]{Fovelle}
A.~Fovelle.
\newblock Hamming graphs and concentration properties in non-quasi-reflexive
  banach spaces.
\newblock {\em Houston J. Math.}, 48(3):539--579, 2022.

\bibitem[JLS96]{JohnsonLindenstraussSchechtman1996GAFA}
W.~Johnson, J.~Lindenstrauss, and G.~Schechtman.
\newblock Banach spaces determined by their uniform structures.
\newblock {\em Geom. Funct. Anal.}, 6(3):430--470, 1996.

\bibitem[Kal07]{Kalton2007}
N.~Kalton.
\newblock Coarse and uniform embeddings into reflexive spaces.
\newblock {\em Q. J. Math.}, 58(3):393--414, 2007.

\bibitem[Kal08]{KaltonSurvey}
N.~Kalton.
\newblock The nonlinear geometry of {B}anach spaces.
\newblock {\em Rev. Mat. Complut.}, 21(1):7--60, 2008.

\bibitem[Kal11]{Kalton2011FundMath}
N.~Kalton.
\newblock Lipschitz and uniform embeddings into {$\ell_\infty$}.
\newblock {\em Fund. Math.}, 212(1):53--69, 2011.

\bibitem[Kal13a]{Kalton2013Examples}
N.~Kalton.
\newblock Examples of uniformly homeomorphic banach spaces.
\newblock {\em Israel J. Math.}, 104:151--182, 2013.

\bibitem[Kal13b]{Kalton2013AsymptoticStructure}
N.~Kalton.
\newblock Uniform homeomorphisms of {B}anach spaces and asymptotic structure.
\newblock {\em Trans. Amer. Math. Soc.}, 365(2):1051--1079, 2013.

\bibitem[KR08]{KaltonRandrianarivony2008}
N.~Kalton and L.~Randrianarivony.
\newblock The coarse {L}ipschitz geometry of {$l_p\oplus l_q$}.
\newblock {\em Math. Ann.}, 341(1):223--237, 2008.

\bibitem[LR18]{LancienRaja2017Houston}
G.~Lancien and M.~Raja.
\newblock Asymptotic and coarse {L}ipshitz structures of quasi-reflexive
  {B}anach spaces.
\newblock {\em Houston J. Math.}, 44(3):927--940, 2018.

\bibitem[Mau03]{MaureyHandbook}
B.~Maurey.
\newblock Type, cotype and {$K$}-convexity.
\newblock In {\em Handbook of the geometry of {B}anach spaces, {V}ol. 2}, pages
  1299--1332. North-Holland, Amsterdam, 2003.

\bibitem[Maz29]{MazurStudia1929}
S.~Mazur.
\newblock Une remarque sur l'homéomorphie des champs fonctionnels.
\newblock {\em Studia Mathematica}, 1(1):83--85, 1929.

\bibitem[MN08]{MendelNaor2008}
M.~Mendel and A.~Naor.
\newblock Metric cotype.
\newblock {\em Ann. of Math. (2)}, 168(1):247--298, 2008.

\bibitem[MP76]{MaureyPisier1976Studia}
B.~Maurey and G.~Pisier.
\newblock S\'{e}ries de variables al\'{e}atoires vectorielles ind\'{e}pendantes
  et propri\'{e}t\'{e}s g\'{e}om\'{e}triques des espaces de {B}anach.
\newblock {\em Studia Math.}, 58(1):45--90, 1976.

\bibitem[Now06]{Nowak2006FundMath}
P.~Nowak.
\newblock On coarse embeddability into {$l_p$}-spaces and a conjecture of
  {D}ranishnikov.
\newblock {\em Fund. Math.}, 189(2):111--116, 2006.

\bibitem[Ros17]{Rosendal2017Sigma}
C.~Rosendal.
\newblock Equivariant geometry of {B}anach spaces and topological groups.
\newblock {\em Forum Math. Sigma}, 5:Paper No. e22, 62, 2017.

\end{thebibliography}
\end{document}